\documentclass{amsart}
\usepackage{amssymb,mathtools}
\usepackage[british]{babel}
\usepackage{enumitem}
\usepackage[latin1]{inputenc}
\usepackage{url}

\newtheorem{theorem}{Theorem}
\numberwithin{theorem}{section}
\newtheorem{corollary}[theorem]{Corollary}
\newtheorem{lemma}[theorem]{Lemma}
\newtheorem{proposition}[theorem]{Proposition}
\theoremstyle{definition}
\newtheorem{definition}[theorem]{Definition}
\newtheorem{remark}[theorem]{Remark}
\newtheorem{example}[theorem]{Example}

\newcommand{\seq}{\mathsf{Seq}}
\newcommand{\wpo}{\mathsf{WPO}}
\newcommand{\po}{\mathsf{PO}}
\newcommand{\lo}{\mathsf{LO}}
\newcommand{\wo}{\mathsf{WO}}

\newcommand{\rng}{\operatorname{rng}}

\newcommand{\supp}{\operatorname{supp}}
\newcommand{\To}{\Rightarrow}

\newcommand{\tw}{{\mathcal{T}W}}

\newcommand{\tr}{\text{Tr}}

\newcommand{\en}{\operatorname{en}}

\title{Normal functions and maximal order types}
\author{Anton Freund and Davide Manca}

\address{Department of Mathematics, Technical University of Darmstadt, Schloss\-garten\-str.~7, 64289~Darmstadt, Germany}
\email{\{freund,manca\}@mathematik.tu-darmstadt.de}

\thanks{Funded by the Deutsche Forschungsgemeinschaft (DFG, German Research Foundation) -- Project number 460597863.}

\begin{document}

\begin{abstract}
Transformations of well partial orders induce functions on the ordinals, via the notion of maximal order type. In most examples from the literature, these functions are not normal, in marked contrast with the central role that normal functions play in ordinal analysis and related work from computability theory. The present paper aims to explain this phenomenon. In order to do so, we investigate a rich class of order transformations that are known as $\mathsf{WPO}$-dilators. According to a first main result of this paper, $\mathsf{WPO}$-dilators induce normal functions when they satisfy a rather restrictive condition, which we call strong normality. Moreover, the reverse implication holds as well, for reasonably well behaved $\mathsf{WPO}$-dilators. Strong normality also allows us to explain another phenomenon: by previous work of Freund, Rathjen and Weiermann, a uniform Kruskal theorem for $\mathsf{WPO}$-dilators is as strong as \mbox{$\Pi^1_1$-}comprehension, while the corresponding result for normal dilators on linear orders is equivalent to the much weaker principle of $\Pi^1_1$-induction. As our second main result, we show~that $\Pi^1_1$-induction is equivalent to the uniform Kruskal theorem for $\mathsf{WPO}$-dilators that are strongly normal.
\end{abstract}

\keywords{Normal function, Maximal order type, Dilator, Uniform Kruskal theorem, Reverse mathematics, Well partial order}
\subjclass[2020]{03B30, 03F15, 03F35, 06A06}

\maketitle

\section{Introduction}

This paper investigates connections between two notions that are central to proof theory and its applications in reverse mathematics: normal functions on the ordinals and the maximal order types of well partial orders.

Recall that a function~$f$ from ordinals to ordinals is normal if it is strictly increasing and continuous at limits, which means that we have $f(\lambda)=\sup\{f(\alpha)\,|\,\alpha<\lambda\}$ for any limit ordinal $\lambda$. Equivalently, the range of~$f$ is a closed and unbounded (club) class of ordinals. When~$f$ is normal, the class of its fixed points is club itself. This fact underlies the definition of the Veblen hierarchy of normal functions: starting with a base function that is typically given by $\varphi_0(\gamma)=\omega^\gamma$, one stipulates that $\varphi_\alpha$ is the increasing enumeration of the club class $\{\gamma\,|\,\varphi_\beta(\gamma)=\gamma\text{ for all }\beta<\alpha\}$. By $\gamma\mapsto\Gamma_\gamma$ one denotes the enumeration of $\{\beta\,|\,\varphi_\beta0=\beta\}$, which is also club.

The Veblen hierarchy is a key ingredient for the ordinal analysis of important mathematical axiom systems. Indeed, it is a famous result of S.~Feferman~\cite{feferman64}~and K.~Sch\"utte~\cite{schuette64} that $\Gamma_0$ is the proof theoretic ordinal of predicative theories, such as the system~$\mathsf{ATR}_0$ from reverse mathematics (see~\cite{simpson09} for background on the latter). Another type of result in reverse mathematics was pioneered by J.-Y.~Girard~\cite{girard87}: he showed that, over $\mathsf{RCA}_0$, arithmetical comprehension is equivalent to the statement that $\beta\mapsto\omega^\beta$ preserves well foundedness (where $\omega^\beta$ is considered as a linear order that is computable relative to~$\beta$; see also the proof by J.~Hirst~\cite{hirst94}). Other~important set existence principles have been characterized in the same way: the transformations \mbox{$\beta\mapsto\varphi_1(\beta)=\varepsilon_\beta$}, $\alpha\mapsto\varphi_\alpha(0)$ and $\gamma\mapsto\Gamma_\gamma$ correspond, respectively, to arithmetical recursion along~$\mathbb N$ and along arbitrary well orders as well as to the principle that any set lies in a countable $\omega$-model of~$\mathsf{ATR}_0$ (see~\cite{rathjen-afshari,marcone-montalban,rathjen-atr,rathjen-weiermann-atr}). Some of these results have been proved both by ordinal analysis and by computability theory, so that they provide a connection between these approaches.

To consider ordinal numbers in a framework such as reverse mathematics, one will often represent them as well orders on suitable systems of terms. For example, the aforementioned order $\omega^\beta$ can be given as the system of Cantor normal forms $\omega^{\beta_0}+\ldots+\omega^{\beta_{n-1}}$, seen as terms with constant symbols $\beta_i<\beta$. These so-called ordinal notation systems are well-motivated, but for larger ordinals the syntactic details are somewhat intricate (cf.~the representation of~$\varphi_\alpha(0)$ in \cite[Section~2]{rathjen-weiermann-atr}). One can significantly reduce the syntactic complexity by considering partial rather than linear orders. We now introduce some terminology that makes this precise. Let us say that a function $f:X\to Y$ between partial orders is a quasi embedding if it reflects the order, i.\,e., if $f(x)\leq_Y f(x')$ entails $x\leq_X x'$. By a linearization of~$X$ we mean a surjective quasi embedding $\alpha\to X$ for a linear order~$\alpha$. A~partial order~$X$ is called a well partial order if any infinite sequence $x_0,x_1,\ldots\subseteq X$ involves an inequality $x_i\leq_X x_j$ for some~$i<j$. It is straightforward to infer that $\alpha$ is a well order for any linearization~$\alpha\to X$. In fact, this property yields one of several equivalent characterizations of well partial orders. Since the equivalence is not provable in~$\mathsf{RCA}_0$ (see~\cite{cholak-RM-wpo}), we insist on the given definition when the base theory is relevant. In a fundamental paper of D.~de Jongh and R.~Parikh~\cite{deJongh-Parikh}, the maximal order type of a well partial order~$X$ has been defined as the ordinal
\begin{equation*}
o(X):=\sup\{\alpha\,|\,\text{there is a linearization }\alpha\to X\},
\end{equation*}
where each well order~$\alpha$ is identified with the isomorphic ordinal number. As shown by de Jongh and Parikh, the given supremum is in fact a maximum, i.\,e., a linearization $o(X)\to X$ does always exist. By an unpublished result of de Jongh (see the introduction of~\cite{schmidt75}), the proof theoretic ordinal $\varphi_1(0)=\varepsilon_0$ of Peano arithmetic coincides with the maximal order type of the collection of binary trees, where an inequality between trees is given by an embedding that respects infima. Classical work of D.~Schmidt~\cite{schmidt-habil-new} (originally from her 1979 \mbox{habilitation}) gives analogous characterizations for a range of larger ordinals. An example for recent work is provided by the thesis of J.~van der Meeren~\cite{meeren-thesis} and his papers with M.~Rathjen and A.~Weier\-mann~\cite{MRW-Veblen,MRW-Bachmann}. In all cases, a complex ordinal notation system is characterized in terms of a partial order that is simpler and `more mathematical'. In addition to their intrinsic interest, these characterizations have important applications in reverse mathematics. Famously, H.~Friedman has combined such a characterization and a result of ordinal analysis to show that predicative axiom systems cannot prove Kruskal's theorem, even in a finitized version (see~\cite{simpson85} and the precise bounds determined by Rathjen and Weiermann~\cite{rathjen-weiermann-kruskal}). We note that this provides a concrete mathematical example for the incompleteness phenomenon from G\"odel's theorems. As a second more recent application, we mention the analysis of Fra\"iss\'e's conjecture for linear orders of finite Hausdorff rank, which is due to A.~Marcone and A.~Montalb\'an~\cite{marcone-montalban-hausdorff}.

To motivate our contribution in the present paper, we take a somewhat closer look at a classical example. Let $\seq(X)$ be the partial order of finite sequences in a given partial order~$X$, where we have $\langle x_0,\ldots,x_{m-1}\rangle\leq_{\seq(X)}\langle y_0,\ldots,y_{n-1}\rangle$ when there is a strictly increasing function $f:\{0,\ldots,m-1\}\to\{0,\ldots,n-1\}$ such that $x_i\leq_X y_{f(i)}$ holds for all~$i<m$. Higman's lemma is the result that $\seq(X)$ is a well partial order whenever the same holds for~$X$. Over $\mathsf{RCA}_0$, this is equivalent to the statement that $\beta\mapsto\omega^\beta$ preserves well foundedness, as shown by S.~Simpson~\cite{simpson-higman}. Due to the aforementioned result of Girard, it follows that Higman's lemma is equivalent to arithmetical comprehension. The precise relation between the partial and the linear case, however, is somewhat intricate: according to~\cite{hasegawa94} we~have
\begin{equation*}
o\big(\seq(X)\big)=\begin{cases}
\omega^{\omega^{o(X)-1}} & \text{if $o(X)$ is finite (where $\omega^{-1}:=0$)},\\
\omega^{\omega^{o(X)+1}} & \text{if $o(X)=\varphi_1(\beta)+n$ for some $\beta$ and $n<\omega$},\\
\omega^{\omega^{o(X)}} & \text{otherwise}.
\end{cases}
\end{equation*}
Let us note that $o(\seq(X))$ does only depend on~$o(X)$. As any ordinal~$\alpha$ is equal to its maximal order type~$o(\alpha)$, we may thus focus on the function $\alpha\mapsto o(\seq(\alpha))$. The analogous point can be made for many examples from the literature (see in particular \cite{meeren-thesis,schmidt-habil-new}). Let us now recall that $\varphi_1$ enumerates the fixed points of the function $\gamma\mapsto\varphi_0(\gamma)=\omega^\gamma$, which are called $\varepsilon$-numbers. It~follows that $\gamma<\varphi_1(\beta)$ entails $\omega^\gamma<\varphi_1(\beta)$, so that we get
\begin{equation*}
\sup\left\{\left.o\big(\seq(\gamma)\big)\,\right|\,\gamma<\varphi_1(\beta)\right\}=\varphi_1(\beta)<\omega^{\omega^{\varphi_1(\beta)+1}}=o\left(\seq\big(\varphi_1(\beta)\big)\right).
\end{equation*}
This means that the function $\alpha\mapsto o(\seq(\alpha))$ is not normal. In his impressive work on ordinal notations, R.~Hasegawa describes this as a ``strange fact" that provides the starting point for his investigation (see~\cite[Section~3]{hasegawa94}). The same \mbox{phenomenon} occurs when we replace $X\mapsto\seq(X)$ by other natural transformations of partial orders, such as $X\mapsto X+X$ or $X\mapsto X\times X$ (see~\cite[Theorems~3.4 and~3.5]{deJongh-Parikh}), a multiset ordering studied by Aschenbrenner and Pong as well as Weiermann (see~\cite{aschenbrenner-pong} and \cite[Theorem~2]{weiermann09}), or different variants of labelled trees (see~\cite{meeren-thesis,rathjen-weiermann-kruskal,schmidt-habil-new}). A~rare case where we do get a normal function is a different order on multi\-sets, as explained in Example~\ref{ex:multiset} below.

To summarize, we have seen that normal functions are central in ordinal analysis but rare in the study of maximal order types, even though the two approaches have important connections. The first aim of the present paper is to give a systematic explanation of this ``strange fact" (taking up the quote by Hasegawa from above). For this purpose, we shall study a rich class of functors on well partial orders, which are called $\wpo$-dilators. These have been introduced in~\cite{frw-kruskal}, in analogy~with Girard's dilators on linear orders (see~\cite{girard-pi2}). Details are recalled in Section~\ref{sect:strong-normality}. Also in the latter, we identify a simple `syntactic' criterion that singles out a class of \mbox{$\wpo$-}dilators that we call strongly normal (as a related but weaker notion of normality has been considered in~\cite{frw-kruskal}). We then prove a first main result of the present paper: for any $\wpo$-dilator~$W$, the assumption that~$W$ is strongly normal is sufficient, and under certain conditions necessary, to ensure that $\alpha\mapsto o(W(\alpha))$ is a normal function on the ordinals. We will see that strong normality is a rather restrictive condition, which one expects to fail in most natural cases (even though the aforementioned multiset construction does provide a relevant example). As promised, this explains why normal functions are so rare in connection with maximal order types.

In Section~\ref{sect:uniform-Kruskal} we show that the notion of strong normality sheds light on another situation where the cases of partial and linear orders have not matched up so~far. Each $\wpo$-dilator~$W$ that is normal in the sense of~\cite{frw-kruskal} (i.\,e., not necessarily strongly normal) gives rise to a certain partial order~$\mathcal T W$. The statement that $\mathcal T W$ is a well partial order for any normal $\wpo$-dilator~$W$ is called the uniform Kruskal theorem, as several variants of the original theorem arise for specific~$W$. By a result of Freund, Rathjen and Weiermann~\cite{frw-kruskal}, the uniform Kruskal theorem is equivalent to the strong set existence principle of $\Pi^1_1$-comprehension, over $\mathsf{RCA}_0$ extended by the chain antichain principle. Somewhat analogous to the construction of~$\mathcal T W$,~each dilator~$D$ on linear orders is associated with a certain linear order~$\vartheta D$. The statement that $\vartheta D$ is well founded for any dilator~$D$ on linear orders is also equivalent to  $\Pi^1_1$-comprehension, as shown in~\cite{freund-equivalence,freund-computable}. However, when we restrict to~$D$ that are normal, we obtain an equivalence with the much weaker principle of \mbox{$\Pi^1_1$-induction} along~$\mathbb N$, now over~$\mathsf{ACA}_0$ (see~\cite{freund-single-fixed-point} and compare~\cite{freund-rathjen_derivatives}). As our second main result, we show that $\Pi^1_1$-induction along~$\mathbb N$ is equivalent to the uniform Kruskal theorem for strongly normal $\wpo$-dilators. This completes the picture and confirms strong normality as `the right' characterization of normal functions in the partial case.

\section{Strong normality}\label{sect:strong-normality}

In this section, we recall the definition of $\wpo$-dilator and introduce the notion of strong normality. We then discuss how the strong normality of a $\wpo$-dilator~$W$ relates to the normality of the function $\alpha\mapsto o(W(\alpha))$.

Let $\po$ be the category with the partial orders as objects and the quasi embeddings (order reflecting functions) as morphisms. A quasi embedding $f:X\to Y$ is called an embedding if it does also preserve the order. We say that a functor $W:\po\to\po$ preserves embeddings if $W(f):W(X)\to W(Y)$ is an embedding whenever the same holds for~$f:X\to Y$. Let us write $[\cdot]^{<\omega}$ for the finite subset functor on the category of sets, with
\begin{align*}
[X]^{<\omega}&:=\text{`the set of finite subsets of~$X$'},\\
[f]^{<\omega}(a)&:=\{f(x)\,|\,x\in a\}\quad\text{for $f:X\to Y$ and $a\in[X]^{<\omega}$}.
\end{align*}
The forgetful functor to the underlying set will be left implicit, e.\,g., when we consider the natural transformation $W\Rightarrow[\cdot]^{<\omega}$ in the following definition. Conversely, a subset of an ordered set will often be considered as a suborder. We shall write $\rng(f)=\{f(x)\,|\,x\in X\}$ for the range of a function~$f:X\to Y$. The following notion has been introduced in~\cite{frw-kruskal}. It is closely related to Girard's definition of dilators on linear orders~\cite{girard-pi2} (consider \cite[Remark~2.2.2]{freund-thesis} to see the precise connection).

\begin{definition}\label{def:podilator}
A $\po$-dilator consists of a functor $W:\po\to\po$ that preserves embeddings and a natural transformation $\supp:W\To[\cdot]^{<\omega}$ such that we have
\begin{equation*}
\supp_Y(\sigma)\subseteq\rng(f)\quad\To\quad\sigma\in\rng(W(f)),
\end{equation*}
for any embedding~$f:X\to Y$ and all~$\sigma\in W(Y)$. If, in addition, $W(X)$ is a well partial order whenever the same holds for~$X$, then $W$ is called a $\wpo$-dilator.
\end{definition}

The implication in the definition will be called the support condition. One should note that the converse implication follows from naturality. To explain our choice of morphisms, we recall that a linearization is a special kind of quasi embedding, as observed in the introduction. The condition that~$W$ preserves embeddings ensures that it is determined by its restriction to (morphisms between) finite \mbox{orders}. For example, to evaluate $\sigma\leq\tau$ in~$W(X)$, we consider the inclusion $\iota:a\hookrightarrow X$ of the finite set $a:=\supp_X(\sigma)\cup\supp_X(\tau)$. The support condition yields $\sigma=W(\iota)(\sigma_0)$ and $\tau=W(\iota)(\tau_0)$ for suitable $\sigma_0,\tau_0\in W(a)$. To determine the inequality in~$W(X)$, it is now enough to evaluate $\sigma_0\leq\tau_0$ in~$W(a)$, since $W(\iota)$ is an embedding. As shown in~\cite{frw-kruskal}, one can exploit this observation to represent $\po$-dilators in the framework of reverse mathematics, parallel to the case of Girard's dilators on linear orders. The following notion has also been introduced in~\cite{frw-kruskal}, where it was required for the construction of the partial order~$\mathcal T W$ that we have mentioned in the introduction.

\begin{definition}\label{def:normal}
A $\po$-dilator~$W$ is called normal if we have
\begin{equation*}
\sigma\leq_{W(X)}\tau\quad\To\quad\text{for any $x\in\supp_X(\sigma)$ there is an $x'\in\supp_X(\tau)$ with $x\leq_X x'$},
\end{equation*}
for any partial order~$X$ and all $\sigma,\tau\in W(X)$.
\end{definition}

For dilators on linear orders, the analogous condition characterizes continuity at limit ordinals (cf.~the work of P.~Aczel~\cite{aczel-phd,aczel-normal-functors} and the flowers of Girard~\cite{girard-pi2} as well as the reversal provided by~\cite[Theorem~1]{afrw-sharps}). In~\cite{frw-kruskal} it seemed reasonable to keep the term `normal' in the context of partial orders, even though the connection with normal functions on the ordinals is lost, as the following example shows.

\begin{example}\label{ex:seq}
The transformation $X\mapsto\seq(X)$ that we have considered in the introduction can be turned into a $\wpo$-dilator in the following way: If $f:X\rightarrow Y$ is a quasi embedding, we define $\seq(f):\seq(X)\rightarrow \seq(Y)$ as the quasi embedding given by the clause
\begin{equation*}
    \seq(f)(\langle x_0,...,x_{m-1}\rangle)= \langle f(x_0),...,f(x_{m-1})\rangle.
\end{equation*}
Moreover, we define the family of functions $\supp_X:\seq(X)\rightarrow [X]^{<\omega}$ by setting 
\begin{equation*}
    \supp_X(\langle x_0,...,x_{m-1}\rangle)=\{x_0,...,x_{m-1}\}.
\end{equation*}
It is straightforward to verify that the given functions form a $\wpo$-dilator. We recall that an inequality $\sigma=\langle x_0,...,x_{m-1}\rangle\le_{\seq(X)}\langle y_0,...,y_{n-1}\rangle=\tau$ is witnessed by a strictly increasing function $h:\{0,...,m-1\}\rightarrow\{0,...,n-1\}$ with $x_i\leq_X y_{h(i)}$ for all~$i<m$. The latter entails that any $x\in\supp_X(\sigma)$ is majorized by some element $y\in\supp_X(\tau)$. Therefore, the $\wpo$-dilator $\seq$ is normal. At the same time, the map $\alpha\mapsto o(\seq(\alpha))$ is not continuous, as we have seen in the introduction.
\end{example}

It will be convenient to consider a somewhat stronger notion of normality, which is analogous to a condition that Aczel~\cite{aczel-phd,aczel-normal-functors} has introduced in the linear case. We shall write $\operatorname{Id}_{\mathcal C}$ for the identity functor on a category~$\mathcal C$.

\begin{definition}\label{def:Aczel-normal}
A $\po$-dilator~$W$ is called Aczel-normal if it comes with a natural transformation $\mu:\operatorname{Id}_\po\Rightarrow W$ such that we have
\begin{equation*}
x\leq_X y\text{ for some }y\in\supp_X(\sigma)\quad\Leftrightarrow\quad \mu_X(x)\leq_{W(X)}\sigma,
\end{equation*}
for any partial order~$X$ and all $x\in X$ and $\sigma\in W(X)$.
\end{definition}

The $\po$-dilator $\seq$ from Example~\ref{ex:seq} is Aczel-normal with $\mu_X(x)=\langle x\rangle$. More generally, when $W(X)$ consists of $X$-labelled structures, we can typically take $\mu_X(x)$ to be a single point with label~$x$. Let us confirm the following.

\begin{lemma}\label{lem:Aczel-normal-normal}
Any Aczel-normal $\po$-dilator is normal in the sense of Definition~\ref{def:normal}.
\end{lemma}
\begin{proof}
Let $W$ be an Aczel-normal $\po$-dilator, and consider a partial order $X$ and elements $\sigma,\tau\in W(X)$ such that $\sigma\le_{W(X)}\tau$. For any $x\in\supp_X(\sigma)$, we have that $\mu_X(x)\le_{W(X)}\sigma\le_{W(X)}\tau$. Hence, there is a $y\in\supp_X(\tau)$ such that $x\le_X y$.
\end{proof}

We continue with two further fundamental properties.

\begin{lemma}\label{lem:Aczel-normal-basic}
The following holds whenever $W$ is an Aczel-normal $\po$-dilator:
\begin{enumerate}[label=(\alph*)]
\item We have $\supp_X(\mu_X(x))=\{x\}$ for any partial order~$X$ and all $x\in X$.
\item Each function $\mu_X:X\to W(X)$ is an order embedding.
\end{enumerate}
\end{lemma}

\begin{proof}
As usual, we write $n$ for the linear order $\{0,...,n-1\}\subseteq\omega$. First, we prove that $\supp_1(\mu_1(0))=1$. Aiming for a contradiction, assume that the support is empty, and hence it is included in the range of the empty embedding ${l:\emptyset\rightarrow 1}$. By the support condition, we find $\mu_0\in W(\emptyset)$ with $\mu_1(0)=W(l)(\mu_0)$. For ${f_i:1\ni 0\mapsto i\in 2}$, naturality of $\mu$ entails $\mu_2\circ f_i=W(f_i)\circ \mu_1$. In light of $f_0\circ l=f_1\circ l$, we get
\begin{equation*}
    \mu_2(0)=\mu_2(f_0(0))=W(f_0)(\mu_1(0))=W(f_0\circ l)(\mu_0)=W(f_1\circ l)(\mu_0)=\mu_2(1),
\end{equation*}
and hence that $\mu_2$ is not injective. This contradicts the fact that the components of $\mu$ are quasi embeddings.
 Now let us consider an arbitrary partial order $X$ and compute the support of $\mu_X(x)$ for some $x\in X$. We define $\iota$ as the embedding $1\ni 0\mapsto x\in X$. Naturality of $\mu$ entails $\supp_X(\mu_X(x))=\supp_X(W(\iota)(\mu_1(0)))$. Since $\supp$ is natural as well, the latter equals $[\iota]^{<\omega}\circ\supp_1(\mu_1(0))=\{x\}$. This proves statement (a). Moreover, for $x,y\in X$ note that $\mu_X(x)\le_{W(X)}\mu_X(y)\Leftrightarrow x\le_X y$ follows from Definition~\ref{def:Aczel-normal} and the fact that $\supp_X(\mu_X(y))=\{y\}$, so we get  statement (b) as well.
\end{proof}

In the case where well partial orders are preserved, we get the following extensional consequence. Let us recall that $o(X)$ denotes the maximal order type of~$X$, as explained in the introduction.

\begin{proposition}\label{prop:strictly-incr}
If $W$ is an Aczel-normal $\wpo$-dilator, the map $\alpha\mapsto o(W(\alpha))$ on ordinals is strictly increasing.
\end{proposition}
\begin{proof}
Given ordinals $\alpha<\beta$, we write $\iota:\alpha\hookrightarrow \beta$ for the inclusion map. Since~$W$ preserves embeddings, we have that $\rng(W(\iota))$ is an isomorphic copy of $W(\alpha)$ inside~$W(\beta)$.  The aforementioned result by de~Jongh and Parikh~\cite{deJongh-Parikh} guarantees the existence of a linearization $l:o(W(\alpha))\rightarrow W(\alpha)$. We call $\le_l$ the linear ordering on $\rng(W(\iota))$ that is induced by $l$ in the obvious way. 
Now consider the relation
\begin{equation*}
    {\le'}={\le_l}\cup\big\{(\sigma,\mu_\beta(\alpha))\,|\,\sigma \in \rng(W(\iota))\cup\{\mu_\beta(\alpha)\}\big\}.
\end{equation*}
Note that the naturality of $\supp$ entails that $\mu_\beta(\alpha)\not\in\rng(W(\iota))$, because its support is $\{\alpha\}\not\subseteq\alpha=\rng(\iota)$. Therefore, $\le'$ is a linear order with order type $o(W(\alpha))+1$. We claim that $\le '$ extends the restriction of $\le_{W(\beta)}$ to $\rng(W(\iota))\cup\{\mu_\beta(\alpha)\}$. If that is true, one finds a partial ordering on $W(\beta)$ which extends both $\le_{W(\beta)}$ and $\le'$, thus proving that the maximal order type of $W(\beta)$ is at least $o(W(\alpha))+1$. The details can be found in Lemma 2.2 of \cite{deJongh-Parikh}. To prove our claim, first we observe that $\le_l$ is compatible with $\le_{W(\beta)}$ restricted to $\rng(W(\iota))$ because $l$ is a linearization of $W(\alpha)$ and $W(\iota)$ is an embedding. It remains to check that the remaining inequalities in $\le'$ are compatible with $\le_{W(\beta)}$, i.e.\ that $\mu_\beta(\alpha)\not\le_{W(\beta)}\sigma$ for all $\sigma\in\rng(W(\iota))$. This is indeed the case, because $\mu_\beta(\alpha)\le_{W(\beta)}\sigma$ entails $\alpha\le x$ for some $x\in\supp_\beta(\sigma)$, and on the other hand $\sigma\in \rng(W(\iota))$ implies $\supp_\beta(\sigma)\subseteq\alpha$ by naturality of $\supp$.
\end{proof}

Our next aim is to identify a structural condition that characterizes those Aczel-normal $\wpo$-dilators for which $\alpha\mapsto o(W(\alpha))$ is a normal function, i.\,e., continuous at limit ordinals. Let us recall that the latter can fail, as seen in Example~\ref{ex:seq}. Given a partial order~$\leq$, we write $x<x'$ to abbreviate the conjunction of $x\leq x'$ and $x\neq x'$ (or equivalently of $x\leq x'$ and $x'\not\leq x$). We will see that the following condition provides the desired characterization.

\begin{definition}\label{def:strongly-normal}
An Aczel-normal $\po$-dilator~$W$ is strongly normal if we have
\begin{equation*}
x<_X y\text{ for all }x\in\supp_X(\sigma)\quad\Rightarrow\quad\sigma<_{W(X)}\mu_X(y),
\end{equation*}
for any partial order~$X$ and all $y\in X$ and $\sigma\in W(X)$.
\end{definition}

It is instructive to observe that the implication in Definition~\ref{def:strongly-normal} follows from the equivalence in Definition~\ref{def:Aczel-normal} when $W(X)$ is a linear order. Also note that the implication in Definition~\ref{def:strongly-normal} is equivalent to
\begin{equation*}
    \text{there is }y\in\supp_X(\tau)\text{ with }x<_Xy\text{ for all }x\in\supp_X(\sigma)\quad\Rightarrow\quad\sigma<_{W(X)}\tau.
\end{equation*}
This could be taken as an alternative definition of strong normality for $\po$-dilators that are normal but not Aczel-normal. However, the assumption that~$W$ is Aczel-normal will play an important role in the proof of Corollary~\ref{cont_to_SN} below.

\begin{example}\label{ex:multiset}
A finite multiset of elements of a set $X$ is a map $\sigma:X\rightarrow\mathbb N$ such that $\sigma(x)=0$ holds for all but finitely many $x\in X$. It is common to denote such a map by the expression $[x_0,\ldots,x_{k-1}]$ where each $x_i\in X$ occurs $\sigma(x_i)$-many times, and to write $x\in\sigma$ to signify $\sigma(x)\neq 0$. As these notations suggest, informally one wants to think of multisets as collections where the same element can occur more than once. The usual operations on sets are extended to multisets as follows: \begin{equation*}
    \sigma\cap\tau:X\ni x\mapsto \min\{\sigma(x),\tau(x)\},
\end{equation*}
\begin{equation*}
    \sigma\smallsetminus\tau:X\ni x\mapsto \max\{0,\sigma(x)-\tau(x)\}.
\end{equation*}
If $X$ is a partial order, we let $M(X)$ be the set of finite multisets of elements of~$X$. To turn $M$ into a $\po$-dilator, we first stipulate
\begin{equation*}
   \sigma\le_{M(X)}\tau\quad\Leftrightarrow\quad\text{for each }x\in\sigma\smallsetminus\tau\text{ there is a }y\in \tau\smallsetminus\sigma\text{ with }x<_Xy.
\end{equation*}
This ordering is a classical tool in the context of term rewriting (see~\cite{dershowitz-manna}). Furthermore, when $f:X\rightarrow Y$ is a quasi embedding and hence injective, we declare that the function $M(f):M(X)\to M(Y)$ is defined by 
\begin{equation*}
    M(f)(\sigma):Y\ni y\mapsto\begin{cases}\sigma(x)& \text{if $y=f(x)$,}\\
    0 & \text{if $y\notin\rng(f)$,}
    \end{cases}
\end{equation*}
or, more informally, $M(f)([x_0,...,x_{k-1}])=[f(x_0),...,f(x_{k-1})]$. One readily checks that $M(f)$ is a quasi embedding with respect to the multiset ordering defined above, and it is an embedding when the same holds for $f$. 
Finally, the support of a multiset $\sigma\in W(X)$ is defined as the set $\{x\in X|\,\sigma(x)\neq 0\}$. 
One can verify that what we get is indeed a $\po$-dilator. Moreover, there is a quasi embedding from~$(M(X),\le_{M(X)})$ into $\seq(X)$ with the order from Higman's lemma. Therefore, $M(X)$ is a well~partial order whenever the same holds for $X$. It is not hard to see that $M$ is strongly normal with $\mu_X:X\to M(X)$ given by $\mu_X(x)=[x]$. In fact, we already get $\sigma<_{M(X)}\tau$ when any $x\in\supp_X(\sigma)$ admits a $y\in\supp_X(\tau)$ with $x<_X y$. Concerning maximal order types, we have $o(M(X))=\omega^{o(X)}$ whenever~$X$ is a well partial order, as shown in~\cite{Weiermann-termination,MRW-Veblen}. In particular, the map $\alpha\mapsto o(M(\alpha))$ is a normal function. Due to the general Theorem~\ref{thm:strongly-normal-equiv} below, this is indeed guaranteed by the fact that~$M$ is a strongly normal $\wpo$-dilator.
\end{example}

Given an element $x$ of a partial order~$X$, we write $L_X(x)$ or just $L(x)$ for the suborder that consists of all $y\in X$ with $x\not\leq_X y$. The next result follows from work of de Jongh and Parikh~\cite{deJongh-Parikh} (see~\cite[Lemma~2.1]{montalban-linearizations} for an explicit statement).

\begin{lemma}\label{otype}
The maximal order type of a well partial order~$X$ satisfies
\begin{equation*}
    o(X)=\sup\{o(L_X(x))+1\,|\,x\in X\}.
\end{equation*}
\end{lemma}

Let us now derive that strong normality is sufficient and in many cases necessary to ensure continuity at limit stages. We will later identify a condition that allows to remove the restriction to linear orders in the following statement~(i).

\begin{theorem}\label{thm:strongly-normal-equiv}
Let $W$ be an Aczel-normal $\wpo$-dilator. Then the following statements are equivalent:
\begin{enumerate}[label=(\roman*)]
    \item The implication from Definition~\ref{def:strongly-normal} holds whenever $X$ is a linear order.
    \item If $\lambda$ is a limit ordinal and we have $\sigma \in W(\lambda)$, then there exists an $\alpha<\lambda$ such that $\tau\in L_{W(\lambda)}(\sigma)$ implies $\supp_\lambda(\tau)\subseteq \alpha$.
    \item The map $\alpha\mapsto o(W(\alpha))$ is a normal function.
\end{enumerate}
\end{theorem}
\begin{proof}
To show that~(i) implies~(ii), we consider an element $\sigma \in W(\lambda)$ for a limit~$\lambda$. If for some $\tau \in W(\lambda)$ we have $\sigma\not \le_{W(\lambda)}\tau$, then $\sigma\not \le_{W(\lambda)}\mu_\lambda(\gamma)$ for all $\gamma$ in $\supp_\lambda(\tau)$, as otherwise we would get $\sigma\le\mu_\lambda(\gamma)\le\tau$. The strong normality property~(i) and linearity entail that each $\gamma\in\supp_\lambda(\tau)$ must be smaller than or equal to the maximum of $\supp_\lambda(\sigma)$. Given that $\lambda$ is a limit, we can conclude by picking an $\alpha<\lambda$ that is larger than this maximum. To see that (ii) implies~(iii), let us first recall that $\alpha\mapsto o(W(\alpha))$ is strictly increasing by Proposition~\ref{prop:strictly-incr}. It remains to derive continuity at a limit ordinal~$\lambda$. Given $\sigma\in W(\lambda)$, pick an $\alpha$ as in~(ii) and consider the inclusion $\iota:\alpha\hookrightarrow\lambda$. In view of $\rng(\iota)=\alpha$, the support condition entails $L(\sigma)\subseteq\rng(W(\iota))$, which yields $o(L(\sigma))\leq o(W(\alpha))$ and hence
\begin{equation*}
o(L(\sigma))+1\leq o(W(\alpha))+1\leq o(W(\alpha+1)).
\end{equation*}
By Lemma~\ref{otype} we get $o(W(\lambda))\leq\sup\{o(W(\beta))\,|\,\beta<\lambda\}$, as needed for~(iii). Finally, we show that (iii) fails when~(i) does. In this case, we have a linear order~$X$ as well as elements $y\in X$ and $\sigma\in W(X)$ with $x<_X y$ for all $x\in \supp_X(\sigma)$ and yet~$\sigma\not<_{W(X)}\mu_X(y)$. Note that we even get $\sigma\not\leq_{W(X)}\mu_X(y)$, as Lemma~\ref{lem:Aczel-normal-basic} yields $\supp_X(\mu_X(y))=\{y\}$ and hence $\sigma\neq\mu_X(y)$. Due to the support condition, we may assume $X=\supp_X(\sigma)\cup\{y\}$ and indeed $X=|\supp_X(\sigma)|+1=n+1$ with $y=n$, as $X$ is linear. Let $\lambda$ be a limit for which $\alpha<\lambda$ implies $o(W(\alpha))<\lambda$. To see that such a $\lambda$ exists, note that we have $\gamma\leq o(W(\gamma))$ since $\gamma\mapsto o(W(\gamma))$ is strictly increasing. We can can now take $\lambda=\sup_{n<\omega}\lambda_n$ with $\lambda_0=0$ and $\lambda_{n+1}=o(W(\lambda_n))+1$. For all $\alpha<\lambda$ with $\alpha\geq n$, we define an embedding $f_\alpha:n+1\rightarrow\lambda$ by
\begin{equation*}
    f_\alpha(j)=\begin{cases}j & \text{ if $j<n$,}\\
                                \alpha &\text{ if $j=n$.}
    \end{cases} 
\end{equation*}
We have $W(f_\alpha)(\sigma)\not\leq_{W(\lambda)}W(f_\alpha)(\mu_{n+1}(n))$ for all~$\alpha$ as above. The naturality of $\mu$ entails that $W(f_\alpha)(\mu_{n+1}(n))=\mu_\lambda(f_\alpha(n))=\mu_\lambda(\alpha)$. Moreover, for the embedding $l:n\to n+1$ with $l(i)=i$, the support condition entails that we have $\sigma=W(l)(\sigma_0)$ for some $\sigma_0\in W(n)$. Now, for any $n\leq\alpha<\lambda$, we note that $f_n\circ l=f_\alpha\circ l$ entails
\begin{equation*}
    \sigma^*:=W(f_n)(\sigma)=W(f_n\circ l)(\sigma_0)=W(f_\alpha\circ l)(\sigma_0)=W(f_\alpha)(\sigma).
\end{equation*}
 We learn that $\{ \mu_\lambda(\alpha):n\leq\alpha<\lambda\}\subseteq L(\sigma^*)$ is an increasing sequence of length~$\lambda$, since $\mu_\lambda$ is an embedding. This yields
 \begin{equation*}
     \sup\{o(W(\alpha))\,|\,\alpha<\lambda\}\leq\lambda\leq o(L(\sigma^*))<o(W(\lambda)),
 \end{equation*}
 so that (iii) does indeed fail.
\end{proof}

In the rest of this section, we identify a condition under which statement~(i) from the previous theorem extends from linear to partial orders, i.\,e., under which we can show that strong normality is necessary. As mentioned in the introduction, we are most interested in $\wpo$-dilators~$W$ for which $o(W(X))$ does only depend on the maximal order type of~$X$. This makes it natural to focus on the case where the underlying set of $W(X)$ does not depend on the order on~$X$. We can capture this case via the following condition, which has already been studied in~\cite{freund-Bachmann-Howard-derivatives}.

\begin{definition}
A $\po$-dilator is flat if the support condition holds for all quasi embeddings, i.e., if we have
\begin{equation*}
\supp_Y(\sigma)\subseteq\rng(f)\quad\To\quad\sigma\in\rng(W(f)),
\end{equation*}
for any quasi embedding~$f:X\to Y$ and all~$\sigma\in W(Y)$.
\end{definition}

We cannot expect that $\sigma\leq_{W(X)}\tau$ will entail $W(f)(\sigma)\leq_{W(Y)}W(f)(\tau)$ whenever the map $f:X\to Y$ is a quasi embedding. At the same time, the condition that $f$ must be an embedding may appear unnecessarily strong. Indeed, we have already mentioned the intuition that the elements of~$W(X)$ are structures or graphs with labels from~$X$. In this setting, it makes sense to assume that $\sigma\leq_{W(X)}\tau$ is witnessed by a map that sends each label $x\in\supp_X(\sigma)$ to a label $y\in\supp_X(\tau)$ with $x\leq_X y$. The following definition puts the focus on these crucial inequalities.

\begin{definition}\label{graph-like}
A $\po$-dilator~$W$ is called graph-like if we have
\begin{equation*}
    \sigma\leq_{W(X)}\tau\quad\To\quad W(f)(\sigma)\leq_{W(Y)}W(f)(\tau)
\end{equation*}
for any quasi embedding $f:X\to Y$ such that $x\leq_X y$ entails $f(x)\leq_Y f(y)$ for all elements $x\in\supp_X(\sigma)$ and $y\in\supp_X(\tau)$.
\end{definition}

A large number of natural $\po$-dilators from the literature are flat and graph-like, including those from Examples~\ref{ex:seq} and~\ref{ex:multiset}. As promised, we can now formulate an elegant consequence of Theorem~\ref{thm:strongly-normal-equiv}.

\begin{corollary}\label{cont_to_SN}
Consider an Aczel-normal $\wpo$-dilator~$W$. If the latter is flat and graph-like, then the following are equivalent:
\begin{enumerate}[label=(\roman*)]
\item The $\wpo$-dilator~$W$ is strongly normal.
\item The map $\alpha\mapsto o(W(\alpha))$ is a normal function on the ordinals.
\end{enumerate}
\end{corollary}
\begin{proof}
It suffices to show that statement~(i) from Theorem~\ref{thm:strongly-normal-equiv} implies strong normality, under the present hypotheses. Aiming at the contrapositive, we assume that $W$ is not strongly normal. We then have a partial order~$Y$ as well as elements $\tau\in W(Y)$ and $y\in Y$ with $x<_Y y$ for all $x\in\supp_Y(\tau)$ but still~$\tau\not<_{W(Y)}\mu_Y(y)$. We in fact get $\tau\not\leq_{W(Y)}\mu_Y(y)$ and may assume $Y=\supp_Y(\tau)\cup\{y\}$, as before. Pick a linearization of~$\supp_Y(\tau)$, i.\,e., a surjective quasi embedding $l_0:n\to\supp_Y(\tau)$ for~$n=|\supp_Y(\tau)|$. Note that we get a quasi embedding $l:n+1\to Y$ by setting $l(n):=y$ and $l(i):=l_0(i)$ for $i<n$. Given that $W$ is flat, we obtain $\tau=W(l)(\sigma)$ for some $\sigma\in W(n+1)$. The naturality of supports yields $\supp_{n+1}(\sigma)=n$. To refute statement~(i) from Theorem~\ref{thm:strongly-normal-equiv}, we show $\sigma\not<_{W(n+1)}\mu_{n+1}(n)$. If the last inequality did hold, the assumption that $W$ is graph-like would yield
\begin{equation*}
    \tau=W(l)(\sigma)\leq_{W(Y)} W(l)(\mu_{n+1}(n))=\mu_Y(l(n))=\mu_Y(y),
\end{equation*}
which contradicts an assumption from above.
\end{proof}

One can show the following by adapting the previous proof in a rather straightforward way.

\begin{remark}
The implication (ii)$\Rightarrow$(i) also holds when the dilator $W$ is flat and satisfies the following slightly modified version of Definition \ref{graph-like}: We have
\begin{equation*}
    \sigma<_{W(X)}\mu_X(y)\quad\To\quad W(f)(\sigma)<_{W(Y)}\mu_Y(f(y))
\end{equation*}
for any quasi embedding $f:X\rightarrow Y$ such that $x\le_Xy$ entails $f(x)\le_Yf(y)$ for all elements $x\in \supp_X(\sigma)$. To better appreciate the similarity, recall that we have $\mu_Y(f(y))=W(f)(\mu_X(y))$. It is worth noting that this alternative condition, although less intuitive than the one from Definiton \ref{graph-like}, is automatically verified by all strongly normal dilators. To see that this is the case, consider a quasi embedding $f$ as described above, and assume that $\sigma<_{W(X)}\mu_X(y)$. Since $W$ is Aczel-normal, we have that any $x\in\supp_X(\sigma)$ is strictly smaller than $y$. Then, for any such $x$, we get $f(x)<_Y f(y)$, because $f$ is injective. This means that $f(y)$ is strictly greater than all the  $x'$ in $\supp_Y(W(f)(\sigma))$, as $\supp_Y(W(f)(\sigma))=[f]^{<\omega}(\supp_X(\sigma))$. By strong normality, we conclude that $W(f)(\sigma)<\mu_Y(f(y))$.
\end{remark}

\section{Normality and the uniform Kruskal theorem}\label{sect:uniform-Kruskal}

As mentioned in the introduction, it was shown by Freund, Rathjen and Weier\-mann~\cite{frw-kruskal} that $\Pi^1_1$-comprehension is equivalent to a uniform Kruskal theorem for normal $\wpo$-dilators, over a weak base theory from reverse mathematics. In this section, we show that the much weaker principle of $\Pi^1_1$-induction along~$\mathbb N$ is equivalent to the uniform Kruskal theorem for $\wpo$-dilators that are strongly~normal.

Let us briefly discuss the representation of $\po$-dilators in reverse mathematics. The key idea is that dilators are determined by their restrictions to finite orders, as Girard \cite{girard-pi2} had observed in the linear case. In order to show this, we first fix a collection $\po_0\subseteq\po$ that contains exactly one isomorphic copy $|a|$ of each finite partial order~$a$. Let us also fix isomorphisms $\en_a:|a|\rightarrow a$. For any partial order~$X$, each element $\sigma\in W(X)$ can be identified with the unique pair $(a,\sigma_0)$ such that we have $a=\supp_X(\sigma)$ and $\sigma=W(\iota_a^X\circ\en_a)(\sigma_0)$. Here $\iota_a^X$ denotes the inclusion of $a$ into $X$, and the support condition guarantees the existence of an appropriate $\sigma_0$.

The identification described above preserves the ordering in the following sense: consider $\sigma,\tau\in W(X)$ represented as pairs $(a,\sigma_0),(b,\tau_0)$.
Then one has 
\begin{equation*}
    \sigma\le_{W(X)}\tau\quad\Leftrightarrow\quad W(|\iota_a^{a\cup b}|)(\sigma_0)\le_{W(|a\cup b|)}W(|\iota_b^{a\cup b}|)(\tau_0),
\end{equation*}
where, for a quasi embedding $f:a\rightarrow b$ with $a,b$ finite, we define $|f|:|a|\rightarrow |b|$ as the unique quasi embedding such that $f\circ \en_a=\en_b\circ \,|f|$.  Moreover, if $\sigma$ is represented by $(a,\sigma_0)$, naturality of the support entails 
\begin{equation*}
    a=\supp_{X}(\sigma)=[\iota^X_a\circ\en_a]^{<\omega}\circ\supp_{|a|}(\sigma_0)
\end{equation*}
and hence $\supp_{|a|}(\sigma_0)=|a|$. So $(|a|,\sigma_0)$ is contained in the trace, defined as 
\begin{equation*}
    \tr(W)=\{(a,\sigma)|\,a\in \po_0\text{ and }\sigma \in W(a)\text{ with } \supp_a(\sigma)=a\}.
\end{equation*}
As promised, $\po$-dilators can thus be represented in $\mathsf{RCA}_0$ (relative to a fixed choice of isomorphisms $\en_a:|a|\to a$). Full details can be found in~\cite{frw-kruskal}, which is also the source of the following key notion.

\begin{definition}\label{def:fixpoint}
A Kruskal fixed point of a $\po$-dilator $W$ is a pair $(X,\kappa)$, where~$X$ is a partial order and $\kappa:W(X)\rightarrow X$ is a bijection such that
\begin{equation*}
    \kappa(\sigma)\le_{X}\kappa(\tau) \quad \Leftrightarrow \quad \sigma\le_{W(X)}\tau \text{ or there is a } y\in\supp_X(\tau)\text{ with }\kappa(\sigma)\le_X y
\end{equation*}
holds for all $\sigma,\tau\in W(X)$. Moreover, $(X,\kappa)$ is initial if for any other Kruskal fixed point $(X',\kappa')$ there is a unique quasi embedding $f:X\rightarrow X'$ with $f\circ\kappa=\kappa'\circ W(f)$.
\end{definition}

As usual, the universal property ensures that initial Kruskal fixed points are unique up to isomorphism. Concerning existence, the following construction in~$\mathsf{RCA}_0$ has been given in~\cite{frw-kruskal}. First, generate recursively the collection of all terms of the form $\circ(a,\sigma)$, where $\sigma$ is the second component of some pair in $\tr(W)$ and $a$ is a finite (possibly empty) set of previously constructed terms of the same form. We declare that the length of a term is given by
$l(\circ(a,\sigma))=1+\sum_{s\in a}2\cdot l(s)$. We can use simultaneous recursion on these lengths to define a subset~$\tw$ of the indicated collection of terms and a binary relation~$\le_\tw$ on this subset. In the following, the condition that $\le_\tw$ should be a partial order on~$a$ is included to ensure that $|a|$ is defined, even though part~(a) of Theorem~\ref{thm:fixpoint} will mean that it is redundant.
\begin{itemize}
    \item By recursion on $l(r)$, we declare that $r=\circ(a,\sigma)$ is an element of $\tw$ if we have $a\subseteq\tw$, the restriction of $\le_\tw$ to $a$ is a partial order, and we have $(|a|,\sigma)\in\tr(W)$ with respect to this order.
    \item By recursion on $l(s)+l(t)$, we declare that $s=\circ(a,\sigma)\le_\tw\circ(b,\tau)=t$ with $s,t\in\tw$ holds if and only if we have $s\le_\tw r$ for some $r\in b$ or it is the case that $a\cup b$ is partially ordered by $\le_\tw$ and we have
    \begin{equation*}
       W(|\iota_a^{a\cup b}|)(\sigma)\le_{W(|a\cup b|)}W(|\iota_b^{a\cup b}|)(\tau).
    \end{equation*}
\end{itemize}
The construction above is available for arbitrary $\po$-dilators. In the normal case, Proposition 3.6 and Theorem 3.8 of \cite{frw-kruskal} tell us the following. 
\begin{theorem}[$\mathsf{RCA}_0$]\label{thm:fixpoint}
Let $W$ be a normal $\po$-dilator.
\begin{enumerate}[label=(\alph*)]
    \item The relation $\le_\tw$ is a partial order on $\tw$.
    \item We obtain an initial Kruskal fixed point $(\tw,\kappa)$ of~$W$ by stipulating that we have $\kappa(\sigma)=\circ(a,\sigma_0)$ with $a=\supp_{\tw}(\sigma)$ and $\sigma=W(\iota_a^\tw\circ\en_a)(\sigma_0)$.
\end{enumerate}
\end{theorem}
In the sequel, we shall prove that the principle of induction along the natural numbers for $\Pi^1_1$-formulas is equivalent to the statement that $\le_\tw$ is a well partial order whenever $W$ is a strongly normal $\wpo$-dilator. We begin by showing the forward implication. Define by recursion the height of a term $\circ(a,\sigma)\in \tw$ as 
\begin{equation*}
    h(\circ(a,\sigma))=\max(\{0\}\cup\{h(s)+1\,|\,s\in a\}).
\end{equation*}
The following is a kind of converse to Lemma~3.5 of~\cite{frw-kruskal}.
\begin{lemma}[$\mathsf{RCA}_0$]\label{height}
Consider a strongly normal $\po$-dilator W. Then
\begin{equation*}
    h(s)<h(t)\quad\Rightarrow\quad s<_\tw t
\end{equation*}
holds for all $s,t\in \tw$.
\end{lemma}
\begin{proof}
We argue by induction on the build-up of~$t$. In view of Theorem~\ref{thm:fixpoint}(b) we may write $s=\kappa(\sigma)=\circ(a,\sigma_0)$ and $t=\kappa(\tau)=\circ(b,\tau_0)$. Given $h(s)<h(t)$, there must be a $t'\in b=\supp_{\tw}(\tau)$ with $h(t)=h(t')+1$ and hence $h(s')<h(s)\leq h(t')$ for all~$s'\in a=\supp_{\tw}(\sigma)$. The latter entails $s'<_\tw t'$ by induction, so that strong normality yields $\sigma<_{W(\tw)}\mu_{\tw}(t')\leq_{W(\tw)}\tau$. From Definition \ref{def:fixpoint} we know that $\kappa$ preserves the order. We thus get $s=\kappa(\sigma)\le_\tw\kappa(\tau)=t$, which is a strict inequality because $h(s)<h(t)$ entails $s\neq t$.\end{proof}

We now derive the first part of the promised equivalence.

\begin{proposition}[$\mathsf{RCA}_0$]\label{prop:ind-to-wpo}
Assume induction over~$\mathbb N$ for all $\Pi^1_1$-formulas. If $W$ is a strongly normal $\wpo$-dilator, then $\tw$ is a well partial order.
\end{proposition}

\begin{proof}
Let us recall that a sequence $s_0,s_1,\ldots$ in $\tw$ is bad if there are no $i<j$ with $s_i\leq_\tw s_j$. In this situation, Lemma \ref{height} implies
\begin{equation*}
    0\le i<j\quad\To\quad h(s_j)\le h(s_i).
\end{equation*}
Knowing this, we only need to argue that $\{s\in\tw\,|\,h(s)<n\}$ contains no bad sequence, for all $n\in\mathbb N$. We do so by induction on $n$. For the induction step, assume towards a contradiction that $s_0,s_1,\ldots\subseteq\tw$ is a bad sequence with~$h(s_0)=n$. By Theorem~\ref{thm:fixpoint} we may write $s_i=\kappa(\sigma_i)=\circ(a(i),\sigma_i')$ with $\sigma_i=W(\iota_{a(i)}^\tw\circ\,\en_{a(i)})(\sigma_i')$. Due to the induction hypothesis, the collection
\begin{equation*}
X:=\bigcup_{i\in\mathbb N}a(i)\subseteq \{s\in\tw\,|\,h(s)<n\}
\end{equation*}
is a well partial order. The same holds for $W(X)$, since~$W$ is a $\wpo$-dilator. We thus find indices $i<j$ with
\begin{equation*}
    W(\iota_{a(i)}^X\circ\en_{a(i)})(\sigma_i')\leq_{W(X)}W(\iota_{a(j)}^X\circ\en_{a(j)})(\sigma_j').
\end{equation*}
Now compose with the embedding $W(\iota_X^\tw):W(X)\to W(\tw)$, to get
\begin{equation*}
    \sigma_i=W(\iota_X^\tw\circ\iota_{a(i)}^X\circ\en_{a(i)})(\sigma_i')\leq_{W(\tw)}W(\iota_X^\tw\circ\iota_{a(j)}^X\circ\en_{a(j)})(\sigma_j')=\sigma_j.
\end{equation*}
As $\kappa$ preserves the order, we can conclude $s_i=\kappa(\sigma_i)\le_{\tw}\kappa(\sigma_j)=s_j$, which contradicts the assumption that our sequence was bad.
\end{proof}

Our next objective is to establish the opposite implication. This will rely on a previous result on the linear case. We write $\lo$ for the category of linear orders and embeddings. This is a full subcategory of~$\po$, as a quasi embedding is an embedding when the range is linear. An $\lo$-dilator consists of a functor~$D:\lo\Rightarrow\lo$ and a natural transformation $\supp:D\To[\cdot]^{<\omega}$ that validate the support condition from Definition~\ref{def:podilator}. If $D(X)$ is well founded for every well order~$X$, then~$D$ is called a $\wo$-dilator. Let us note that the $\wo$-dilators coincide with the original dilators of Girard~\cite{girard-pi2} (as explained in Remark~2.2.2 of~\cite{freund-thesis}). Given a functor $W:\po\to\po$, we write $W\restriction\lo:\lo\to\po$ for its restriction to the category of linear orders. In the following definition, we also view~$D:\lo\to\lo$ as a functor from~$\lo$ to~$\po$, by implicitly post-composing with $\lo\hookrightarrow\po$.

\begin{definition}\label{def:quasi-emb-dil}
Consider a $\lo$-dilator $D$ and a $\po$-dilator $W$. A quasi embedding from $D$ into $W$ is a natural transformation $\nu:D\To W\upharpoonright\lo$.
\end{definition}

Let us recall that the components of natural transformations are morphisms, so that $\nu_X:D(X)\to W(X)$ is a quasi embedding for each linear order~$X$. In~\cite{frw-kruskal} it is shown how $\nu$ is determined by its action on the subcategory $\lo_0\subseteq\lo$ of finite linear orders, which allows for a representation in reverse mathematics.

We say that an $\lo$-dilator~$D$ is Aczel-normal if it comes with a natural family of embeddings~$\mu_X:X\to D(X)$ that validate the equivalence from Definition~\ref{def:Aczel-normal}. In the linear case, this is equivalent to the condition that we have
\begin{equation*}
    \sigma<_{D(X)}\mu_X(y)\quad\Leftrightarrow\quad x<_X y\text{ for all }x\in\supp_X(\sigma)
\end{equation*}
for all $\sigma\in D(X)$ and $y\in X$. This reveals that our Aczel-normal $\wo$-dilators coincide with the normal dilators from~\cite{freund-rathjen_derivatives}. As noted in the previous section, the notion goes back to work of Aczel~\cite{aczel-phd,aczel-normal-functors} and relates to Girard's flowers~\cite{girard-pi2}.

A central argument in~\cite{frw-kruskal} concerns a quasi embedding $\nu:D\To W\restriction\lo$ of an arbitrary $\lo$-dilator~$D$ into a $\po$-dilator~$W$ that is normal and in fact \mbox{Aczel-normal}. For the case where $D$ itself is Aczel-normal, we shall now show that $\nu$ factors over a quasi embedding $\nu^*:D\To W^*$ such that~$W^*$ is strongly normal.

\begin{definition}[$\mathsf{RCA}_0$]\label{def:W*}
Consider an Aczel-normal $\po$-dilator~$W$ and the transformations $\supp:W\Rightarrow[\cdot]^{<\omega}$ and $\mu:\operatorname{Id}_\po\To W$ that come with it. We define the following structure (see the next lemma for verifications):
\begin{itemize}
\item For each partial order~$X$, we define $W^*(X)$ as the partial order with the same underlying set as~$W(X)$ and the order relation given by
\begin{equation*}
    \sigma\leq_{W^*(X)}\tau\quad\Leftrightarrow\quad\begin{cases}
    \text{we have }\sigma\leq_{W(X)}\tau\text{ or there is a }y\in\supp_X(\tau)\\\text{with }x<_X y\text{ for all }x\in\supp_X(\sigma).
    \end{cases}
\end{equation*}
\item We declare that the functions
\begin{equation*}
    W^*(f):W^*(X)\to W^*(Y),\quad\supp^*_X:W^*(X)\to[X]^{<\omega},\quad \mu^*_X:X\to W^*(X)
\end{equation*}
coincide with $W(f)$, $\supp_X$ and $\mu_X$, respectively, for any partial orders~$X,Y$ and any quasi embedding~$f:X\to Y$.
\end{itemize}
\end{definition}

In the following we check the expected properties.

\begin{lemma}[$\mathsf{RCA}_0$] The previous definition yields a $\po$-dilator~$W^*$ that is strongly normal. When~$W$ is a $\wpo$-dilator, the same holds for~$W^*$.
\end{lemma}
\begin{proof}
Let us first recall that we have $\supp_X(\mu_X(x))=\{x\}$ due to Lemma~\ref{lem:Aczel-normal-basic}. As a preliminary observation, we can infer that $\mu^*$ validates Definition~\ref{def:Aczel-normal} (including the fact that $\mu^*_X(x)\leq_{W^*(X)}\mu^*_X(y)$ entails $x\leq_X y$), i.\,e.~the conditions for being Aczel-normal. By the proof of Lemma~\ref{lem:Aczel-normal-normal}, it follows that~$W^*$ validates the normality condition from Definition~\ref{def:normal}. Based on this fact, we now show that $W^*(X)$ is a partial order for any given partial order~$X$. Reflexivity is clearly inherited from~$W(X)$. Concerning antisymmetry, we consider the case where $\sigma\leq_{W^*(X)}\tau$ holds because we have a $y\in\supp_X(\tau)$ with $x<_X y$ for all~$x\in\supp_X(\sigma)$. If we also had $\tau\leq_{W^*(X)}\sigma$, the implication from Definition~\ref{def:normal} would yield an $x\in\supp_X(\sigma)$ with $y\leq_X x$, which would lead to a contradiction. To establish transitivity, we assume $\rho \le_{W^*(X)}\sigma$ and $\sigma\le_{W^*(X)}\tau$. If both inequalities do also hold in $W(X)$, then we get $\rho\le_{W(X)}\tau$ and hence $\rho\le_{W^*(X)}\tau$. In order to cover the remaining cases, we first assume that there is a $y\in\supp_X(\sigma)$ with $x<_X y$ for all $x\in\supp_X(\rho)$. The implication from Definition~\ref{def:normal} yields a $z\in\supp_X(\tau)$ with~$y\leq_X z$. We can conclude that $x<_X z$ holds for all $x\in\supp_X(\rho)$, so that we indeed get~$\rho\leq_{W^*(X)}\tau$. A similar argument applies when $\sigma\le_{W^*(X)}\tau$ holds because there is a $z\in\supp_X(\tau)$ with $y<_X z$ for all~$y\in\supp_X(\sigma)$. Next, we consider a quasi embedding $f:X\to Y$ and assume $W^*(f)(\sigma)\le_{W^*(Y)}W^*(f)(\tau)$. In the crucial case, we have a $y\in\supp_Y(W(f)(\tau))$ with $x<_Y y$ for all~$x\in\supp_Y(W(f)(\sigma))$. Due to the naturality of supports, we may write $y=f(y')$ with $y'\in\supp_X(\tau)$. For any $x'\in\supp_X(\sigma)$ we have $f(x')\in\supp_Y(W(f)(\sigma))$, so that we get $f(x')<_Y f(y')$ and hence $x'<_X y'$. This yields $\sigma\leq_{W^*(X)}\tau$, as needed to show that~$W^*(f)$ is a quasi embedding. A similar argument shows that $W^*(f)$ is an embedding when the same holds for~$f$. The other conditions in Definition~\ref{def:podilator} do not concern the order relation and are therefore inherited from~$W$. Hence $W^*$ is indeed a $\po$-dilator. We have already considered the condition that makes it Aczel-normal. The strong normality condition from Definition~\ref{def:strongly-normal} is satisfies by construction (as Lemma~\ref{lem:Aczel-normal-basic} ensures $\supp_X(\mu_X(y))=\{y\}$). Finally, $W^*(X)$ is a well partial order when the same holds for $W(X)$, as the identity $W^*(X)\to W(X)$ is a quasi embedding.
\end{proof}

Let us now prove the aforementioned factorization result.

\begin{lemma}[$\mathsf{RCA}_0$]
Consider a quasi embedding $\nu:D\To W\restriction\lo$, where $D$ is an $\lo$-dilator and $W$ is an Aczel-normal $\po$-dilator. If $D$ is also Aczel-normal, we get a quasi embedding $\nu^*:D\To W^*\restriction\lo$.
\end{lemma}
\begin{proof}
Given that $W^*(X)$ and $W(X)$ have the same underlying set, we stipulate that $\nu^*_X$ is the same map as $\nu_X$, for each linear order~$X$. Naturality is immediate. The task is to show that $\nu^*_X:D(X)\to W^*(X)$ is a quasi embedding. Let us assume that we have $\nu^*_X(\sigma)\leq_{W^*(X)}\nu^*_X(\tau)$. In the crucial case, we have a $y\in{\supp^W_X}\circ\nu_X(\tau)$ with $x<_X y$ for all $x\in{\supp^W_X}\circ\nu_X(\tau)$. Here $\supp^W_X$ is the support function that comes with~$W$, while $\supp^D_X$ will denote the one that comes with~$D$. Lemma~4.2 of~\cite{frw-kruskal} ensures ${\supp^W_X}\circ\nu_X=\supp^D_X$. This means that we have $y\in\supp^D_X(\tau)$ as well as $x<_X y$ for all $x\in\supp^D_X(\sigma)$. Given that $D$ is Aczel-normal, we can conclude that we have $\sigma<_{D(X)}\mu_X(y)\leq_{D(X)}\tau$ (see the explanations after Definition~\ref{def:quasi-emb-dil}).
\end{proof}

 We can finally prove the remaining implication in our main result.
 
\begin{theorem}
The following are equivalent over $\mathsf{ACA}_0$:
\begin{enumerate}[label=(\roman*)]
\item Induction along~$\mathbb N$ is available for all $\Pi^1_1$-formulas.
\item If $W$ is any strongly normal $\wpo$-dilator, then its initial Kruskal fixed point~$\tw$ is a well partial order.
\end{enumerate}
\end{theorem}
\begin{proof}
In Proposition~\ref{prop:ind-to-wpo} we have seen that~(i) implies~(ii) over $\mathsf{RCA}_0$. For the converse, we rely on Theorem~3.14 from~\cite{freund-single-fixed-point}: it tells us that~(i) follows from the statement that each Aczel-normal $\wo$-dilator~$D$ admits an embedding~\mbox{$D(X)\to X$} for some well order~$X$. Let us note that the base theory $\mathsf{ACA}_0$ is inherited from the cited result. We claim that one can take~$X$ to be the Bachmann-Howard fixed point~$\vartheta D$ that has been constructed in~\cite{freund-computable}. As shown in the latter, we then have a function $\vartheta:D(X)\to X$ with the following properties:
\begin{enumerate}
    \item We get $\vartheta(\sigma)<_X\vartheta(\tau)$ if $\sigma<_{D(X)}\tau$ and $x<_X\vartheta(\tau)$ for all $x\in\supp_X(\sigma)$.
    \item We have $y<_X\vartheta(\tau)$ for all $y\in\supp_X(\tau)$.
\end{enumerate}
In general, the function~$\vartheta$ is no embedding, due to the side condition in~(1). To apply the aforementioned result from~\cite{freund-single-fixed-point}, we now show that $\vartheta$ is an embedding when~$D$ is Aczel-normal. Let us write $\mu^D:\operatorname{ID}_\lo\To D$ for the natural transformation that witnesses this property. We assume $\sigma<_{D(X)}\tau$. To conclude $\vartheta(\sigma)<_X\vartheta(\tau)$ by~(1) and~(2), we note that any $x\in\supp_X(\sigma)$ admits a $y\in\supp_X(\tau)$ with~$x\leq_X y$. Indeed, we would otherwise get $\tau<_{D(X)}\mu^D_X(x)\leq_{D(X)}\sigma$ since $D$ is Aczel-normal, as in the previous proof. It remains to show that~(ii) implies the statement that $\vartheta D$ is well founded for any normal $\wo$-dilator~$D$. For this purpose, we need only produce a quasi embedding of $\vartheta D$ into the initial Kruskal fixed point~$\tw$ of some strongly normal $\wpo$-dilator~$W$. By Theorem~4.5 of~\cite{frw-kruskal}, such a quasi embedding can be obtained from a quasi embedding of~$D$ into~$W$. In Section~5 of the same paper, it is shown how to produce a quasi embedding $\nu:D\To W_D\restriction\lo$ for a $\wpo$-dilator~$W_D$ that is normal but not necessarily strongly normal. Below, we construct a natural transformation~$\mu^W:\operatorname{ID}_\po\To W_D$ that makes $W_D$ Aczel-normal. By the two previous lemmas, this will yield the desired quasi embedding~$\nu^*:D\To W_D^*\restriction\lo$ for a strongly normal $\wpo$-dilator~$W_D^*$. To explain the construction of~$\mu^W$, we recall that~$W_D(X)$ consists of the pairs $(u,\sigma)$ such that $u:[u]=\{0,\ldots,[u]-1\}\to X$ with~$[u]\in\mathbb N$ is a finite quasi embedding and $([u],\sigma)$ lies in the trace of~$D$ (see Definition~5.4 of~\cite{frw-kruskal}). We note that $(1,\mu^D_1(0))\in\tr(D)$ holds essentially due to Lemma \ref{lem:Aczel-normal-basic}~(a). For a partial order $X$ and an element $x\in X$, we define $u^X_x:1\to X$ by $u^X_x(0):=x$. Let us now consider
\begin{equation*}
 \mu^W_X:X\to W_D(X)\quad\text{with}\quad\mu^W_X(x):=(u^X_x,\mu^D_1(0)).
\end{equation*}
The given functions are natural due to $f\circ u^X_x=u^Y_{f(x)}$ for $f:X\to Y$, as the reader can confirm by considering Definition~5.4 of~\cite{frw-kruskal}. In the notation that is used in the same definition, functions $h\in\operatorname{Hig}(u^X_x,u)$ correspond to values $i=h(0)<[u]$ with~$x\leq_X u(i)$. In view of $D(h)(\mu^D_1(0))=\mu^D_{[u]}(h(0))$, the definition thus yields
\begin{equation*}
    \mu^W_X(x)\leq_{W_D(X)}(u,\sigma)\quad\Leftrightarrow\quad\text{there is }i<[u]\text{ with }x\leq_X u(i)\text{ and }\mu^D_{[u]}(i)\leq_{D([u])}\sigma.
\end{equation*}
For $(u,\sigma)=\mu^W_X(y)$, the right side amounts to $x\leq_X u^X_y(0)=y$, which shows that $\mu^W_X$ is an embedding. Since $(u,\sigma)\in W_D(X)$ has support~$\rng(u)=\{u(i)\,|\,i\in[u]\}$, the condition for $W_D$ to be Aczel-normal (see Definition~\ref{def:Aczel-normal})~is
\begin{equation*}
    \mu^W_X(x)\leq_{W_D(X)}(u,\sigma)\quad\Leftrightarrow\quad\text{there is }y\in\rng(u)\text{ with }x\leq_X y.
\end{equation*}
To see that the right sides of the previous equivalences amount to the same, we note that the condition $\mu^D_{[u]}(i)\leq_{D([u])}\sigma$ is redundant. Indeed, it follows from the assumption that $D$ is Aczel-normal, since $(u,\sigma)\in W_D(X)$ requires $([u],\sigma)\in\tr(D)$, which means that $\sigma\in D([u])$ has support~$[u]\ni i$.
\end{proof}

\bibliographystyle{amsplain}
\bibliography{Normal_Max-Otypes}

\end{document}